\documentclass[leqno,12pt]{article}
\usepackage{inputenc}
\usepackage[english]{babel}
\usepackage{amsmath}
\usepackage{amssymb}
\usepackage{amsthm}
\usepackage{setspace}
\date{}
\newtheorem{thm}{Theorem}
\newtheorem{lem}[thm]{Lemma}
\newtheorem{prop}[thm]{Proposition}
\newtheorem{cor}[thm]{Corollary}
\newtheorem{que}[thm]{Question}

\theoremstyle{definition}
\newtheorem{ex}[thm]{\it Example}
\newtheorem{rem}[thm]{\it Remark}

\numberwithin{equation}{section}

\DeclareMathOperator{\ap}{Ap}
\DeclareMathOperator{\minap}{minAp}
\DeclareMathOperator{\maxap}{maxAp}

\title{Numerical semigroups with large embedding dimension satisfy Wilf's conjecture}

\begin{document}

\onehalfspacing

\author{Alessio Sammartano\thanks{{\em email} alessiosammartano@gmail.com}}

\maketitle

\begin{abstract}
\noindent 
We give an affirmative answer to Wilf's conjecture for numerical semigroups satisfying
$ 2 \nu \geq m$, 
where $\nu$ and $m$ are respectively the embedding dimension and the multiplicity of a semigroup.
The conjecture is also proved when $ m \leq 8$ and when the semigroup is generated by a generalized arithmetic sequence.

\medskip


\end{abstract}

\section*{Introduction}

A classical problem in additive number theory is the \emph{Diophantine Frobenius Problem},
also known as money-changing problem:
given $\nu$ coprime positive integers $g_1, \ldots, g_\nu$ determine the largest integer $f$ which is not representable as a linear combination of $g_1, \ldots, g_\nu$ with coefficients in  $ \mathbb{N}$.
The problem, introduced by Sylvester in \cite{Syl} for the case $\nu=2$, 
has been widely studied in literature;
the monograph \cite{RA} gathers a lot of results on the topic. 

It is natural to study the problem in the context of numerical semigroups, i.e. submonoids of the additive monoid of the natural numbers.
It is indeed possible to provide formulas linking the Frobenius number of a semigroup $f(S)$ to its other invariants. 
With regards to this problem, in 1978 Wilf posed in \cite{Wi} the following question:

\begin{que}\label{conj}
Let $S$ be a numerical semigroup with Frobenius number $f(S)$, embedding dimension $\nu(S)$
and let $n(S)= |S \cap [0,f]|$.
Is it true that $ f(S)+1 \leq n(S)\nu(S)$?
\end{que}

The conjecture is still open, although an affirmative answer has been given for various partial cases  (see \cite{Ba}, \cite{DM} and \cite{FGH}).
Moreover, some computations made in \cite{BA} strengthen our convinction in a positive answer to the conjecture.
\\

In this paper, 
after  some background on numerical semigroups, 
we find an equivalent form of the conjecture (cf. Proposition \ref{equiv}, Remark \ref{ali}).
In particular, 
while results known so far rely heavily on the particular hypotheses (for example symmetric, almost symmetric, three-generatd or maximal embedding dimension  semigroups),
we try to develop a more general method to attack the problem.
We show that it is affirmatively answered by semigroups whose embedding dimension is large with respect to the multiplicity (cf. Theorem \ref{god}).
Finally, 
we note that the conjecture is also verified by  semigroups with small multiplicity (cf. Corollary \ref{fail}) and by those generated by a generalized arithmetic sequence (cf. Proposition \ref{gener}).

A good reference about numerical semigroups is \cite{RG}.

\section{Preliminaries}

Let $\mathbb{N}$ denote the set of natural numbers, including $0$.
A \emph{numerical semigroup} 
is a submonoid of $( \mathbb{N},+)$ with finite complement in it.
Given a numerical semigroup $S$, 
we define a partial order setting  $s \preceq t$ if there exists an element $u \in S$ such that $s+u=t$.
Each numerical semigroup has a unique minimal system of generators 
$ \{ g_1 < g_2 < \ldots < g_\nu\}$ such that every element $s\in S$ is representable as 
$s= \lambda_1 g_1+ \ldots + \lambda_\nu g_\nu$, 
with $\lambda_i \in \mathbb{N}$.
This set coincides with the set of minimal elements in  $S\setminus\{0\}$ with respect to the partial order $\preceq$.

There are several invariants associated to a numerical semigroup $S$. 
The largest integer not belonging to the semigroup is called \emph{Frobenius number} of $S$ and is denoted by $f=f(S)$;
the number $f(S)+1$ is known as the \emph{conductor} of $S$.
The \emph{multiplicity} is defined as $m=m(S)=\min\{ s \in S, \, s>0\}$;
it is clear that $m=g_1$.
The number of generators $\nu=\nu(S)$ is called \emph{embedding dimension}; 
it is not difficult to see that the inequality $\nu \leq m $ holds.
An integer $ x \in \mathbb{Z}\setminus S$ is called a \emph{pseudo-Frobenius number} if $x+s \in S$ for every $ s \in S, \, s\ne 0$; 
the cardinality of the set of pseudo-Frobenius numbers is known as the \emph{type} of the semigroup and is denoted with $t=t(S)$.
We use the symbol $|X|$ to denote the cardinality of a set $X$.
The number $n(S)$ is defined as $n(S)=\big|\{s \in S,\, s < f\}\big|$.

Throughout the paper we will make an extensive use of an important tool associated to a semigroup $S$,
which is known as $\emph{Ap\'ery set}$ of $S$ and is defined as 
\[\ap(S)=\{ w \in S,\, w-m \notin S\}.\]
This set consists of the smallest elements in $S$ in each class of congruence modulo $m$ (\cite{RG}, Lemma 2.4);
it follows that $|\ap(S)|=m$ and $0 \in \ap(S)$.
We name the elements in increasing order setting $\ap(S)=\{ w_0 < w_1 < \ldots < w_{m-1}\}$;
with this notation,
we always have $w_0=0,\, w_1=g_2$ and $w_{m-1}=f+m$ (\cite{RG}, Proposition 2.12). 
It is useful to consider $\ap(S)\setminus\{0\}$ as a partially ordered set, 
with the partial order $\preceq$ induced by $S$:
we can indeed state some properties of $S$ in terms of this poset, as we see in the next result.
In order to do this, we define the two subsets 
\[ \minap(S)=\{w \in \ap(S)\setminus\{0\},\, w \mbox{ is minimal wrt} \preceq\}\]
\[ \maxap(S)=\{w \in \ap(S)\setminus\{0\},\, w \mbox{ is maximal wrt} \preceq\}.\]

\begin{prop}[\cite{Br}, Lemma 3.2]
Let $S$ be a numerical semigroup, then:
\begin{itemize}
\item[(i)] $\minap(S)=\{g_2,\ldots,g_\nu\}$; 

\item[(ii)] $\maxap(S)=\{w, \, w-m \mbox{ is a pseudo-Frobenius number of S}\}$.
\end{itemize}
\end{prop}
We obtain in particular that
 $|\minap(S)|=\nu-1$ and $|\maxap(S)|=t(S)$.
The next property is an easy consequence of the definitions of $\ap(S)$ and $\preceq$:
\begin{lem}[\cite{FGH}, Lemma 6]\label{easy}
If $w \in \ap(S)$ and $u \preceq w$,
then $u \in \ap(S)$.
\end{lem}

The following inequality  is useful for some particular cases:

\begin{prop}[\cite{FGH}, Theorem 22]\label{where}
Let $S$ be a semigroup with notation as above. 
Then we have $f(S)+1 \leq n(S)(t(S)+1)$.
\end{prop}

As a consequence, 
every semigroup $S$ such that $ t(S)+1 \leq \nu(S)$ satisfies the conjecture.
Moreover, 
the above inequality has been used in the same paper to prove the next result:

\begin{cor}[\cite{FGH}]\label{frag}
If $\nu(S) \leq 3$, then $S$ satisfies Wilf's conjecture.
\end{cor}

We are now able to prove two results which will be used afterwards.

\begin{lem}\label{unfor}
If $m(S) - \nu(S) \leq 2$, then  $S$ satisfies Wilf's conjecture.
\end{lem}
\begin{proof}
Let us distinguish three cases.
\begin{itemize}
\item
If $\nu=m$, 
then $\ap(S)\setminus\{0\}=\maxap(S)=\{g_2,\ldots,g_\nu\}$ and  $t=\nu-1$.

\item
If $\nu=m-1$, 
then $\ap(S)\setminus\{0\}=\{g_2,\ldots,g_\nu,u\}$ with $g_i \preceq u$ for at least one index $i\in \{2, \ldots, \nu\}$;
it follows that $g_i \notin \maxap(S)$ and $t \leq \nu-1$.

\item
If $\nu=m-2$, 
then $\ap(S)\setminus\{0\}=\{g_2,\ldots,g_\nu,u,v\}$
with $u < v$.
Since $u$ and $v$ are not generators, 
we have either $u=g_h+g_i,\,v=g_j+g_k$ for some indexes such that $\{h,i\}\ne \{j,k\},$ or $v=u+g_j$.
In both cases we find in $\ap(S)\setminus\{0\}$ two elements that are not maximal and hence $t \leq \nu -1$.
\end{itemize}
In each case we have $t(S)+1\leq \nu(S)$, 
hence we can apply Proposition \ref{where}.
\end{proof}
The technique used in the last Lemma cannot be generalized to higher values of $m-\nu$,
since the inequality $t+1 \leq \nu$ does not hold in general,
as the following example shows.
\begin{ex}
Let $S=\langle 7,8,10,19\rangle$.
Then $m-\nu=7-4=3$. 
The Ap\'ery set is given by
\[\ap(S)=\{0,8,10,16,18,19,20\}\]
and the maximal elements are $\maxap(S)=\{16,18,19,20\}$, thus $t(S)=\nu(S)=4$ and we cannot apply Proposition \ref{where}.
\end{ex}
\begin{cor}\label{ten}
If $m(S) \leq 6$, then $S$ satisfies Wilf's conjecture.
\end{cor}
\begin{proof}
If $m \leq 6$, 
then either $\nu \leq 3 $ or $m - \nu \leq 2$ and  the thesis follows from Corollary \ref{frag} and Lemma \ref{unfor}.
\end{proof}

\section{Main Results}

Our aim is to develop a method  based on the idea of counting the elements of $S$ in some intervals of length $m$.
Given an integer $k \geq 0$, we define $k$-th \emph{interval} the set
\[I_k=\big[km,(k+1)m-1\big]=\big\{km, km+1, \ldots, (k+1)m-1\big\} \]
and let $n_k=\big|\{s \in S\cap I_k,\, s <f\}\big|$.
We express the conductor of the semigroup in the form $f(S)+1=Lm+\rho$, 
where $1 \leq \rho \leq m$ and
$L= \big\lfloor\frac{f}{m}	\big\rfloor = \big\lfloor\frac{w_{m-1}}{m}	\big\rfloor -1$.
We notice that $L$ is the index of the last interval $I_k$ such that $I_k \not\subseteq S$, 
that is to say the only index such that $f(S) \in I_k$.
The next proposition states basic properties of the $n_k$'s whose proofs are immediate:

\begin{prop}\label{silence}
We have:
\begin{enumerate}
\item[(i)] $1 \leq n_k \leq m-1 $ for $k=0,\ldots, L$;
\item[(ii)] $n_k= | S \cap I_k|$ for $k=0,\ldots, L-1$;
\item[(iii)] $n_{k_1} \leq n_{k_2} $ if $0 \leq k_1 < k_2 \leq L-1$;
\item[(iv)] $n(S) = \sum_{k=0}^L n_k$.
\end{enumerate}
\end{prop}

Now we express Question \ref{conj} in an equivalent form, 
in terms of the quantities introduced so far.

\begin{prop}\label{equiv}
A semigroup $S$ satisfies Wilf's conjecture if and only if 
\begin{equation}\label{formula}
\sum_{k=0}^{L-1}(n_k\nu-m) + (n_L \nu -\rho) \geq 0 .
\end{equation}
\end{prop}
\begin{proof}
Using Proposition \ref{silence} we have the following equivalences:
\begin{eqnarray*}
 f(S)+1 \leq n(S)\nu(S)  & \Leftrightarrow & Lm + \rho  \leq \nu\sum_{k=0}^L n_k\quad  \Big( = \sum_{k=0}^{L-1} n_k \nu + n_L \nu \Big)  \\
\Leftrightarrow \sum_{k=0}^{L-1} m + \rho \leq \sum_{k=0}^{L-1} n_k \nu + n_L \nu & \Leftrightarrow & \sum_{k=0}^{L-1}(n_k\nu-m) + (n_L \nu -\rho) \geq 0 .
\end{eqnarray*}
\end{proof}

\begin{rem}\label{ali}
In order to prove Wilf's conjecture, 
by means of Proposition \ref{equiv}, 
we may compute the number of intervals with a fixed amount of elements of $S$ less than $Lm$, 
and estimate thus the first part of the sum in \ref{formula}.
More precisely, 
if we consider the quantities
\[ \epsilon_j= \big|\big\{ k \in \mathbb{N}, \, | I_k \cap S|=j,\, k =0, \ldots, L-1\big\}\big|,\, \mbox{ with } j\in \{1, \ldots, m-1 \}\]
then, 
by expanding the sum and gathering the terms with the same value of $n_k$,
 we have 
\[ \sum_{k=0}^{L-1}(n_k\nu-m) = \sum_{j=1}^{m-1} \epsilon_j (j \nu -m). \]
\end{rem}

With the intent of the last remark, 
 we define another similar family of numbers:
\[\eta_j=\big| \big\{ k \in \mathbb{N}, \, |I_k\cap S|=j\big\}\big|,\, \mbox{ with } j\in \{1,\ldots,m-1\}\]
$\eta_j$ is the number of intervals $I_k$ with exactly $j$ elements of $S$ 
(not necessarily less than $Lm$).
The next lemma shows how the two families are related:

\begin{lem}\label{sheep}
Under the above notation, 
we have:
\begin{itemize}
\item $\epsilon_j=\eta_j\quad$  for  $ j \ne | I_L \cap S|$;

\item $\epsilon_j=\eta_j-1\quad$ for $ j = | I_L \cap S|$.
\end{itemize}
\end{lem}
\begin{proof}
The thesis is straightforward as the only difference in the two definitions is made by the interval $I_L$.
\end{proof}
The following proposition allows us to express the numbers $\eta_j$ in terms of the Ap\'ery set.

\begin{prop}\label{mas}
For any $j=1,\ldots,m-1,$ we have
$\eta_j=\big\lfloor \frac{w_j}{m}\big\rfloor - \big\lfloor \frac{w_j-1}{m}\big\rfloor$. 
\end{prop}
\begin{proof} Let us fix an interval $I_k$ and $j \in \{1, \ldots,m-1\}$;
we claim that $I_k$ contains at least $j$ elements of $S$ if and only if $w_{j-1} < (k+1)m$.
In this case, 
the interval $I_k$ contains exactly $j$ elements of $S$ if and only if $w_{j-1} < (k+1)m \leq w_{j}$ and the thesis follows by definition of $\eta_j$.
Let us prove the claim. 
Set $I_k\cap S=\{s_1,\ldots, s_p\}$ and define
 $s'_h=\min\{x \in S, x \equiv s_h \pmod{m}\}$   
 for each $h=1, \ldots, p$:
  by the characterization of $\ap(S)$
 it follows that $\{ s'_1, \ldots, s'_p\} \subseteq \ap(S)$;
 moreover $ s'_h \leq s_h<(k+1)m$.
Conversely,
if $w\in \ap(S)$ and $w <(k+1)m$,
then $w+\lambda m \in S\cap I_k$ for a suitable $\lambda \in \mathbb{N}$ and so $w = s'_h$ for some $h\leq p$. 
Therefore,  
 $\{ s'_1, \ldots, s'_p\}$ is the subset of $\ap(S)$ consisting of all the elements
 less than  $(k+1)m$.
Recalling that the elements in $\ap(S)$ are listed in increasing order we may conclude the proof:
\\

$ |I_k\cap S| \geq j \Leftrightarrow p \geq j \Leftrightarrow w_{j-1} \in \{s'_1, \ldots, s'_p\} \Leftrightarrow
 w_{j-1} < (k+1)m.$
\end{proof}

Now we need some technical lemmas that will be necessary in the main theorem of the paper.

\begin{lem}\label{nok}
Let us suppose  $m-\nu \geq 2$.
Then we have
$ \big\lfloor \frac{w_{\nu+1}}{m} \big\rfloor \geq \big\lfloor \frac{w_1}{m} \big\rfloor + \big\lfloor \frac{w_2}{m} \big\rfloor$. 
\end{lem}
\begin{proof}
Since $m-\nu \geq 2$, 
there are two non-zero elements in $\ap(S)$ that are not generators,
that is, two elements in $\ap(S) \setminus \{ 0, g_2, \ldots, g_\nu\}$:
let $u,\,v$ be the smallest of such elements, with $u<v$.
Since $u$ and $v$ are not minimal in $S\setminus\{0\}$ with respect to $\preceq$, 
then $u=u_1+u_2$, $v=v_1+v_2$ with $u_1,\,u_2,\,v_1,\,v_2$ positive elements of $S$;
by Lemma \ref{easy} these elements must belong to the Ap\'ery set and hence we can write
  $ u= w_h + w_i, \, v = w_j+w_k, $ with $h,\,i,\,j,\,k >0$.
Notice that the case $u\preceq v$, i.e. $v=u+w_j$, is not excluded:
 it may occur, under this notation, 
  that $u=w_k$.
For the choice of $u,\,v,$ we have $u \leq w_\nu $ and $v \leq w_{\nu+1} $.
Finally, 
by
$ u= w_h + w_i \geq w_1 + w_1$ and $v = w_j+w_k \geq w_1+ w_2$
we obtain:
\[ w_{\nu+1}\geq v \geq w_1+w_2 \Rightarrow \Big\lfloor \frac{w_{\nu+1}}{m} \Big\rfloor \geq \Big\lfloor \frac{w_1}{m} \Big\rfloor + \Big\lfloor \frac{w_2}{m} \Big\rfloor. \]
\end{proof}

\begin{lem}\label{make}
Let us suppose $m-\nu \geq 2$.
If $ \big\lfloor \frac{w_{m-1}}{m} \big\rfloor = \big\lfloor \frac{w_1}{m} \big\rfloor + \big\lfloor \frac{w_2}{m} \big\rfloor, $
then $n_L \geq 3$.
\end{lem}

\begin{proof}
Set $w_1=q_1 m +r_1,\, w_2=q_2 m + r_2$, where
$q_1= \big\lfloor \frac{w_1}{m} \big\rfloor, \, q_2= \big\lfloor \frac{w_2}{m} \big\rfloor, \, 0<r_1, r_2<m.$
From the previous lemma and the hypothesis we get
\[  \Big\lfloor \frac{w_1}{m} \Big\rfloor + \Big\lfloor \frac{w_2}{m} \Big\rfloor = \Big\lfloor \frac{w_1+w_2}{m} \Big\rfloor = \Big\lfloor \frac{w_{\nu +1}}{m} \Big\rfloor = \Big\lfloor \frac{w_{m-1}}{m} \Big\rfloor = L+1 \]
and hence $(L+1)m \leq w_1 + w_2 \leq w_{\nu+1} \leq w_{m-1} = f+m=(L+1)m + \rho-1$.
By the equality
$  \big\lfloor \frac{w_1}{m} \big\rfloor + \big\lfloor \frac{w_2}{m} \big\rfloor = \big\lfloor \frac{w_1+w_2}{m} \big\rfloor$
we obtain $r_1+r_2 <  m$,
while by the inequality $(L+1)m \leq w_1 + w_2 \leq  (L+1)m+\rho-1$ we obtain $r_1+r_2 \leq \rho-1$, 
and in particular $r_1 < \rho$ and $r_2 < \rho$. 
It follows that $\{Lm,\, w_1+k_1m,\, w_2+k_2m\} \subseteq [Lm,f] \cap S$ for some $k_1,k_2 \in \mathbb{N}$, 
and so $n_L\geq 3$.
\end{proof}

\begin{lem}\label{box}
Let us suppose $m-\nu \geq 3$, 
then $w_2 < f$.
\end{lem}
\begin{proof}
If $w_2 > f$, 
then $w_i > f$ for each $i =2, \ldots, m-1$.
We claim that  $w_2, \ldots, w_{m-1}  \in \maxap(S)$.
Indeed if $w_i \preceq w_j$ for some $j> i \geq 2$,
then there exists $s \in S, \, s>0,$ such that $w_j=w_i+s$.
We have $s\geq m$, $w_i\geq f+1$, hence $w_j\geq f+m+1$ and $w_j-m \in S$,
 in contradiction with $w_j \in \ap(S)$.
Thus the elements $\{w_2,\ldots, w_{m-1}\}$ are pairwise incomparable with respect to $\preceq$.
It follows that, 
for each $j \geq 2$, 
 $w_j \notin \minap(S)$ if and only if $w_1 \preceq w_j$.
But this may occur for at most one index $j$:
if $w_1 \preceq w_j$, then $w_j=w_1+s$, with $s\in S\setminus\{0\}$.
By Lemma \ref{easy}, $s \in \ap(S)\setminus\{0\}$ and the only possibility is $s=w_1$ and hence $w_j=2w_1$. 
Thus at most one element among $\{w_2, \ldots , w_{m-1}\}$ may not be minimal.
We have proved that $m-3 \leq \nu-1$ and so $m - \nu \leq 2$, absurd.
\end{proof}

\begin{lem}\label{mos}
Let us suppose $m-\nu \geq 3$. 
If $n_L=1$, 
then there are two possibilities:
\begin{enumerate}
\item[(1)]  $n_{L-1} \geq 4$;

\item[(2)] $n_{L-1}=3, \, m-\nu=3 \mbox{ and }  \rho \leq m-2$.
\end{enumerate}
\end{lem}
\begin{proof}
By the previous lemma, 
we have $w_1 < w_2 < f$.
Since $n_L=1$ it follows that $\{Lm, \ldots, f\}\cap S = \{Lm\}$ and thus
$w_1 \in I_{h_1},\, w_2 \in I_{h_2}$ for some indexes $h_1,h_2< L$.

Let us suppose $n_{L-1}=|I_{L-1} \cap S|\leq 3$. 
We have that $(L-1)m,\, w_1+k_1m,\, w_2+k_2m \in I_{L-1} \cap S$ for suitable  $k_1,k_2 \in \mathbb{N}$
and so $n_{L-1}=3$.
Recalling the argument of the proof of Proposition \ref{mas}, 
the fact that $n_{L-1}=|I_{L-1}\cap S|=3$ implies $w_3> Lm$;
since $\{Lm, \ldots, f\}\cap S = \{Lm\}$ we actually have $w_3 >f$.

We want to show,
similarly to what we have done within the proof of Lemma \ref{box},
that the only possible non-zero elements in $\ap(S) \setminus \minap(S)$ are
$2w_1,\, w_1+w_2,\, 2w_2$:
in this case we obtain the second assertion $m-\nu = 3$.
If $w \in \ap(S) \setminus\{0\}$, $w \notin \minap(S)$, then we have $w=w_i+w_j$, with $i,\,j >0$
(here we use again Lemma \ref{easy}).
Now if one of the two indexes is greater than 2, 
for example $i > 2$,
by $w_i \geq w_3 \geq f+1$ and $w_j > m$ it follows $w\geq f+m+1$ and $w-m \in S$, 
in contradiction with $w \in \ap(S)$.
Thus $i,\,j \leq 2$ and the only possible non-minimal elements are $\{2w_1, w_1+w_2, 2w_2\}$.

Finally, 
since $w_1+(k_1+1)m,\, w_2+(k_2+1)m \in I_{L}$ and $n_L=1$, we must have 
$f < w_1+(k_1+1)m < (L+1)m $ and $ f < w_2+(k_2+1)m  < (L+1)m$,
hence $f< (L+1)m-2$  and  $\rho \leq m-2$.
\end{proof}

We are ready to prove our main result.

\begin{thm}\label{god}
If $2 \nu(S) \geq m(S)$, then $S$ satisfies Wilf's conjecture.
\end{thm}

\begin{proof}
By Lemma \ref{unfor} we may assume $m - \nu \geq 3$.
We want to proceed as suggested in Remark \ref{ali}:
we will count the intervals with 1 element of $S$ and those with at least 3 elements.
The hypothesis $2\nu \geq m$ allows us to leave out those with 2 elements:
their contribution to the sum in \ref{formula}, 
according to the content of Remark \ref{ali}, 
is $\epsilon_2 (2\nu-m)$ and hence non-negative by hypothesis.

Using Proposition \ref{mas} we find:
\[ \eta_1= \Big\lfloor \frac{w_1}{m} \Big \rfloor - \Big\lfloor \frac{w_0}{m} \Big \rfloor = \Big\lfloor \frac{w_1}{m} \Big \rfloor\]
because $w_0=0$; moreover
\[\sum_{j=3}^{m-1} \eta_j = \sum_{j=3}^{m-1}\Big( \Big\lfloor \frac{w_j}{m} \Big \rfloor - \Big\lfloor \frac{w_{j-1}}{m} \Big \rfloor \Big)= \sum_{j=3}^{m-1}\Big\lfloor \frac{w_j}{m} \Big \rfloor - \sum_{j=2}^{m-2} \Big\lfloor \frac{w_j}{m} \Big \rfloor= \Big\lfloor \frac{w_{m-1}}{m} \Big \rfloor - \Big\lfloor \frac{w_{2}}{m} \Big \rfloor.\]

Applying Remark \ref{ali}, Lemma \ref{sheep} and the above formulas to inequality \ref{formula} we get:
\begin{eqnarray*}
 \sum_{k=0}^{L-1}(n_k\nu-m) + (n_L \nu -\rho)  = \sum_{j=0}^{m-1} \epsilon_j ( j \nu -m ) + (n_L\nu -\rho)\geq\\
   \eta_1 (\nu -m) + \Big(\sum_{j=3}^{m-1} \eta_j -1\Big)(3\nu -m) + (n_L\nu-\rho)=\\
\Big\lfloor \frac{w_1}{m}  \Big\rfloor (\nu-m) + \Big(\Big\lfloor \frac{w_{m-1}}{m} \Big\rfloor - \Big\lfloor \frac{w_2}{m}  \Big\rfloor-1\Big)
(3\nu-m) +  (n_L\nu-\rho) =\\
\Big(\Big\lfloor \frac{w_{m-1}}{m} \Big\rfloor - \Big\lfloor \frac{w_2}{m}  \Big\rfloor- \Big\lfloor \frac{w_1}{m}  \Big\rfloor - 1 \Big)
(3\nu-m) + \Big\lfloor \frac{w_1}{m}  \Big\rfloor (4\nu-2m) + (n_L\nu-\rho). 
\end{eqnarray*}

Let us distinguish now three possible cases.
\begin{itemize}

\item
The equality holds in Lemma \ref{nok}.
\\
By Lemma \ref{make} we have $n_L \geq 3$;
futhermore $(4\nu-2m) \geq 0$ and we can leave it out. 
We obtain
\[ -(3\nu-m)+(n_L\nu-\rho)=(n_L-3)\nu+(m-\rho) \geq 0.\]

\item
The strict inequality holds in Lemma \ref{nok} and $n_L \geq 2$.
\\
We obtain
\[  \Big(\Big\lfloor \frac{w_{m-1}}{m} \Big\rfloor - \Big\lfloor \frac{w_2}{m}  \Big\rfloor- \Big\lfloor \frac{w_1}{m}  \Big\rfloor - 1 \Big)
(3\nu-m) + \Big\lfloor \frac{w_1}{m}  \Big\rfloor (4\nu-2m) + (n_L\nu-\rho) \geq 0 \]
since $n_L \nu - \rho \geq 2 \nu - m \geq 0$ and    all parts considered are non-negative.
\item 
The strict inequality holds in Lemma \ref{nok} and $n_L = 1$.
\\
By Lemma \ref{mos} we have either $n_{L-1} \geq 4$ or $n_{L-1}=3,\,m-\nu=3$ and $ \rho \leq m-2$.
In the first case, 
we need to add $\nu$ to the sum (corresponding to the interval $I_{L-1}$) and we obtain 
\[\nu +\Big\lfloor \frac{w_1}{m}  \Big\rfloor (4\nu-2m) + (\nu-\rho) \geq 2 \nu - \rho \geq 2 \nu - m \geq 0.\]
In the second case
we may assume $m \geq 7$ by Corollary \ref{ten}. 
We obtain:
\begin{eqnarray*}
 \Big\lfloor \frac{w_1}{m} \Big\rfloor (4\nu-2m) + (\nu - \rho) = \Big\lfloor \frac{w_1}{m} \Big\rfloor (4(m-3)-2m) + (m-3 - \rho) \geq \\
\Big\lfloor \frac{w_1}{m} \Big\rfloor(2m-12) + (m-3-m+2) \geq  (2m-12)-1= 2m - 13  >0.
\end{eqnarray*}

\end{itemize}
The inequality is valid in each case and the thesis is thus proved.
\end{proof}

From the previous theorem we immediately get the following result:

\begin{cor}\label{fail}
If $m(S) \leq 8$,
then $S$ satisfies Wilf's conjecture.
\end{cor}
\begin{proof}
If $m\leq 8$ then we get either $\nu \leq 3 $ or $2 \nu \geq m$; 
the thesis follows from Corollary \ref{frag} and  Theorem \ref{god}.
\end{proof}

We conclude our paper showing another class of numerical semigroups satisfying the conjecture,
which is actually independent of most results of the paper.
A semigroup generated by a \emph{generalized arithmetic sequence} is a semigroup of the kind
$S= \langle m, hm+d, hm+2d, \ldots, hm+ ld\rangle$;
for our purpose, we may assume
$\gcd(m,d)=1,\, m \geq 2,\, l \leq m-2 $.
Such semigroups have been studied in \cite{Ma}.
\begin{prop}\label{gener}
If $S$ is a semigroup generated by a generalized arithmetic sequence, 
then $S$ satisfies Wilf's conjecture.
\end{prop}
\begin{proof}
In (\cite{Ma}, Corollary 3.4) the author proved in particular that
$ t(S) = m - \big\lfloor \frac{m-2}{l} \big\rfloor l -1. $
By definition of $ \lfloor \cdot \rfloor$ we have:
\begin{eqnarray*}
 \frac{m-2}{l} < \Big\lfloor \frac{m-2}{l} \Big\rfloor +1 \Rightarrow m-2 < \Big\lfloor \frac{m-2}{l} \Big\rfloor l +l \Rightarrow \\
 t(S) = m - \Big\lfloor \frac{m-2}{l} \Big\rfloor l -1 < l+1 = \nu(S)
 \end{eqnarray*}
and the thesis follows from Proposition \ref{where}.
\end{proof}

\textbf{Acknowledgments.}
I would like to thank warmly Professor Marco D'Anna for some discussion on the subject and for introducing me to numerical semigroups.

\end{document}